\documentclass[12pt]{amsart}
\usepackage{mathrsfs,latexsym,amsfonts,amssymb}
\newtheorem{theorem}{Theorem}[section]
\newtheorem{lemma}[theorem]{Lemma}
\newtheorem{corollary}[theorem]{Corollary}
\newtheorem{question}[theorem]{Question}

\theoremstyle{definition}
\newtheorem{definition}[theorem]{Definition}
\newtheorem{proposition}[theorem]{Proposition}
\newtheorem{example}[theorem]{Example}

\begin{document}
\title[Some topological properties of spaces between the Sorgenfrey and usual topologies on real number]
{Some topological properties of spaces between the Sorgenfrey and usual topologies on real number}

\author{Fucai Lin}
  \address{(Fucai Lin): 1. School of mathematics and statistics,
  Minnan Normal University, Zhangzhou 363000, P. R. China; 2. Fujian Key Laboratory of Granular Computing and Applications,
Minnan Normal University, Zhangzhou 363000, P. R. China}
  \email{linfucai2008@aliyun.com; linfucai@mnnu.edu.cn}

\author{Jiada Li}
\address{(Jiada Li): School of mathematics and statistics,
  Minnan Normal University, Zhangzhou 363000, P. R. China}
\email{840423797@qq.com}

\thanks{This work is supported by the Key Program of the Natural Science Foundation of Fujian Province (No: 2020J02043), the National
Natural Science Foundation of China (Grant No. 11571158), the Institute of Meteorological Big Data-Digital Fujian and Fujian Key Laboratory of Data Science and Statistics.}

%insert keywords
%insert subject class

%\date{\today}
\begin{abstract}
The $H$-space, denoted as $(\mathbb{R}, \tau_{A})$, has $\mathbb{R}$ as its point set and a basis consisting of usual open interval neighborhood at points of $A$ while taking Sorgenfrey neighborhoods at points of $\mathbb{R}$-$A$. In this paper, we mainly discuss some topological properties of $H$-spaces. In particular, we prove that, for any subset $A\subset \mathbb{R}$,

(1) $(\mathbb{R}, \tau_{A})$ is zero-dimensional iff $\mathbb{R}\setminus A$ is dense in $(\mathbb{R}, \tau_{E})$;

(2) $(\mathbb{R}, \tau_{A})$ is locally compact iff $(\mathbb{R}, \tau_{A})$ is a $k_{\omega}$-space;

(3) if $(\mathbb{R}, \tau_{A})$ is $\sigma$-compact, then $\mathbb{R}\setminus A$ is countable and nowhere dense; if $\mathbb{R}\setminus A$ is countable and scattered, then $(\mathbb{R}, \tau_{A})$ is $\sigma$-compact;

(4) $(\mathbb{R}, \tau_{A})^{\aleph_{0}}$ is perfectly subparacompact;

(5) there exists a subset $A\subset\mathbb{R}$ such that $(\mathbb{R}, \tau_{A})$ is not quasi-metrizable;

(6) $(\mathbb{R}, \tau_{A})$ is metrizable if and only if $(\mathbb{R}, \tau_{A})$ is a $\beta$-space.
\end{abstract}

\maketitle
\section{Introduction}
The usual, metric topology on $\mathbb{R}$ is a topological space which coarser than the Sorgenfrey line $\mathbb{S}$, which is a well known space and has been studied extensively. It is well known that Sorgenfrey line has a basis consisting of all half-open intervals of the form $[a, b)$, where $a< b$.
The class of $H$-spaces mentioned in \cite{Chatyrko} is a space between usual topology of real number $\mathbb{R}$ and topology of Sorgenfrey line $\mathbb{S}$, which was described by Hattori in \cite{Braverman}. The $H$-space, denoted as $(\mathbb{R}, \tau_{A})$, has $\mathbb{R}$ as its point set and a basis consisting of usual open interval neighborhood at points of $A$ while taking Sorgenfrey neighborhoods at points of $\mathbb{R}-A$, that is, the topology $\tau_{A}$ is defined as follows:

\smallskip
(1) For each $x\in A$, $\{(x-\varepsilon, x+\varepsilon): \varepsilon>0\}$ is the neighborhood base at $x$, and

\smallskip
(2) for each $x\in \mathbb{R}\setminus A$, $\{[x, x+\varepsilon): \varepsilon>0\}$ is the neighborhood base at $x$.\\
\smallskip
Chatyrko and Hattori first began to study the properties of such spaces, where many interesting results were
obtained, see \cite{Chatyrko1} and \cite{Chatyrko2}. In particular, for any $A\subset\mathbb{R}$, $(\mathbb{R}, \tau_{A})$ is regular, hereditarily
Lindel\"{o}of, hereditarily separable and Baire space. Moreover, for any closed subset $A$ of $\mathbb{R}$,
they proved that $(\mathbb{R}, \tau_{A})$ is homeomorphic to the Sorgenfrey line $\mathbb{S}$ if and only if $A$ is countable.
In 2017, Kulesza in \cite{Chatyrko} made an improvement and a summary on the basis of Chatyrko and Hattori's work, and the author called this kind of spaces as $H$-space, and demonstrated the properties of $H$-space with respect to homeomorphism, functions, completeness and reversibility. In particular, Kulesza proved that $(\mathbb{R}, \tau_{A})$ is homeomorphic to $\mathbb{S}$ if and only if $A$ is scattered, and $(\mathbb{R}, \tau_{A})$  is complete if and only if $\mathbb{R}\setminus A$ is countable, which implies that $(\mathbb{R}, \tau_{\mathbb{P}})$ is complete. Moreover, Bouziad  and Sukhacheva in \cite{BS} gave some characterizations of some topological properties of $(\mathbb{R}, \tau_{A})$, such as each compact subspace being countable, locally compactness and so on. In this paper, we continue the work of Chatyrko and Hattori by proving additional information about the spaces $(\mathbb{R}, \tau_{A})$. The remaining of this paper is organized as follows.

Section 2 is dedicated to outline some concepts and terminologies. In Section 3, we mainly discuss some topological properties of $H$-spaces, such as zero-dimension, $\sigma$-compactness, $k_{\omega}$-property, perfect property, quasi-metrizability and so on. In particular, we give the characterizations of $A$ or $\mathbb{R}\setminus A$ such that $(\mathbb{R}, \tau_{A})$ has topological properties of zero-dimension, $\sigma$-compactness, and $k_{\omega}$-property respectively. Moreover, we show that $(\mathbb{R}, \tau_{A})^{\aleph_{0}}$ is perfectly subparacompact. Further, we discuss some generalized metric properties of $(\mathbb{R}, \tau_{A})$, and prove that there exists a subset $A\subset \mathbb{R}$ such that $(\mathbb{R}, \tau_{A})$ is not quasi-metrizable. In Section 4, we pose some interesting questions about $H$-spaces.

\maketitle
\section{Preliminaries}
 In this section, we introduce the necessary notation and terminologies.
  First of all, let $\mathbb{N}$, $\mathbb{Z}$ and $\mathbb{R}$ denote the sets of all positive
  integers, all integers and all real numbers, respectively. For undefined
  terminologies, the reader may refer to \cite{Chatyrko3} and \cite{Gr}.

\begin{definition}\cite{Chatyrko3} Let $X$ be a topological space.

\smallskip
(1) $X$ is called {\it zero-dimensional} if it has a base of sets that are at the same time open and closed in it.

\smallskip
(2) $X$ is called a {\it Baire space} if every intersection of a countable collection of open dense sets in $X$ is also dense in $X$.

\smallskip
(3) $X$ is called {\it locally compact}, if every point $x$ of $X$ has a compact neighbourhood, i.e., there exists an open set $U$ and a compact set $K$, such that $x\in U\subseteq K$.

\smallskip
(4) $X$ is called a {\it $k_\omega$-space} if there exists a family of countably many compact subsets $\{K_{n}: n\in\mathbb{N}\}$ of $X$ such that each subset $F$ of $X$ is closed in $X$ provided that $F\cap K_n$ is closed in $K_n$ for each $n\in\mathbb{N}$.

\smallskip
(5) $X$ is {\it $\sigma$-compact} if it is the union of countably many compact subsets of $X$.

\smallskip
(6) $X$ is {\it Lindel\"{o}f} if  every open cover of $X$ has a countable subcover.

\smallskip Clearly, each $k_{\omega}$-space is $\sigma$-compact and each $\sigma$-compact is Lindel\"{o}f.
\end{definition}

\begin{definition}\cite{Chatyrko3, Gr}
(1) A space $X$ is {\it subparacompact} if each open cover of $X$ has a $\sigma$-locally finite closed refinement.

\smallskip
(2) A space $X$ is {\it perfect} if each closed subset of $X$ is a $G_{\delta}$ in $X$.

\smallskip
(3) A space $X$ is {\it perfectly subparacompact} if it is perfect and subparacompact.

\smallskip
(4) A space $X$ is {\it weakly $\theta$-refinable} if for each open cover $\mathscr{U}$ of $X$, there exists an open cover $\mathscr{V}=\bigcup_{n=1}^{\infty}\mathscr{V}(n)$ of $X$ which is refines $\mathscr{U}$ and which has the property that if $x\in X$, then there exists an $n\in\mathbb{N}$ such that $x$ belongs to exactly $k$ members of $\mathscr{V}(n)$ for some $k\in\mathbb{N}$.

\smallskip
(5) A family $\mathscr{U}$ of open sets in $X$ is called {\it interior-preserving} if for $\mathscr{F}\subset \mathscr{U}$ and $y\in\bigcap\mathscr{F}$, $\bigcap\mathscr{F}$ is an open neighborhood of $y$.
\end{definition}

\begin{definition}\cite{Chatyrko3}
A family $\mathscr{P}$ of subsets of a space $X$ is a {\it network} for $X$ if for each point $x\in X$ and any neighbhorhood $U$ of $x$ there is an $P\in\mathscr{P}$ such that $x\in P\subset U$. The {\it network weight} of a space $X$ is defined as the smallest cardinal number of the form $|\mathscr{P}|$, where $\mathscr{P}$ is a network for $X$, this cardinal number is denoted by $nw(X)$.
\end{definition}

\begin{definition}\cite{Gr}
Recall that $(X, \tau)$ is a {\it $\beta$-space} if there exists a function $g: \omega\times X\rightarrow \tau$
such that if $x\in g(n, x_{n})$ for every $n\in\omega$ then the sequence $\{x_{n}\}$ has a cluster point in $X$.
\end{definition}

\begin{definition}\cite{Chatyrko2}
Let $A$ be a subset of $\mathbb{R}$ of the real number. Defined the topology $\tau_{A}$ on $\mathbb{R}$ as follows:

\smallskip
(1)\, For each $x\in A, \{(x-\epsilon, x+\epsilon):\epsilon>0\}$ is the neighborhood base at $x$,

\smallskip
(2)\, For each $x\in \mathbb{R}- A, \{[x, x+\epsilon):\epsilon>0\}$ is the neighborhood base at $x$.

\smallskip
Then $(\mathbb{R}, \tau_{A})$ is called {\it $H$-space}. The point $x$ is called an {\it $\mathbb{R}$-point}, if $x\in A$, otherwise, $x$ is called an {\it $\mathbb{S}$-point}.
\end{definition}

Let $\tau_{E}$ and $\tau_{S}$ denote the usual (Euclidean) topology of $\mathbb{R}$ and the topology of the Sorgenfrey line $\mathbb{S}$ respectively.
It is clear that $\tau_{A}=\tau_{E}$ if $A=\mathbb{R}$ and  $\tau_{A}=\tau_{S}$ if $A=\emptyset$. And it is also obvious that $\tau_{E}\subset\tau_{A}\subset\tau_{S}$.

Some topological properties of $(\mathbb{R}, \tau_{E})$ and $(\mathbb{R}, \tau_{S})$ is as the table below:

\begin{table}[!hbt]
  \centering
  \textbf{Some topological properties of $(\mathbb{R}, \tau_{E})$ and $(\mathbb{R}, \tau_{S})$}

  \ \

  \begin{tabular}{|c|c|c|c|}
      \hline
      Number & Property & $(\mathbb{R}, \tau_{E})$ & $(\mathbb{R}, \tau_{S})$\\
      \hline
      1 & metrizable & Yes & No  \\
      \hline
      2 & Hereditarily Separable & Yes & Yes   \\
      \hline
      3 & Normality & Yes & Yes  \\
      \hline
      4 & Lindel\"{o}ff  & Yes  & Yes   \\
      \hline
      5 & Measurable & Yes & No \\
      \hline
      6 & Baire Space & Yes & Yes   \\
      \hline
      7 & Zero-dimension & No  & Yes  \\
      \hline
      8 & Compactness & No & No \\
      \hline
      9 & Countably Compact & No & No  \\
      \hline
      10 & Local Compactness & Yes & No  \\
      \hline
      11 & Sequential Compactness & No & No  \\
      \hline
      12 & Paracompactness & Yes  & Yes  \\
      \hline
      13 & $\sigma$-Compactness & Yes & No  \\
      \hline
      14 & Connectedness & Yes & No  \\
      \hline
      15 & Path Connectedness & Yes & No  \\
      \hline
      16 & Local Connectedness & Yes & No  \\
      \hline
      17 & Every compact subset is countable & No & Yes  \\
      \hline
  \end{tabular}
  \end{table}

  By the table, it is easy to see that, for every subset $A$ of real number, the $H$-space is always a hereditarily separable, paracompact, Lindel\"{o}ff, normal and first-countable space. And we can also know that, the $H$-space is always not a compact, countably compact or sequentially compact space for any subset $A$ of $\mathbb{R}$. According to the \cite[Proposition 2.3]{Chatyrko2}, $H$-space $(\mathbb{R}, \tau_{A})$ is second-countable if and only if $\mathbb{R}-A$ is countable.

\maketitle
\section{main results}
In this section, we mainly discuss some topological properties of $H$-spaces, such as zero-dimension, $\sigma$-compactness, $k_{\omega}$-property, perfect property, quasi-metrizability and so on. we First, we give an obvious lemma.

\begin{lemma}\label{l0}
Let $D$ be a dense subset of $(\mathbb{R}, \tau_{A})$. Then $D$ is dense in $(\mathbb{R}, \tau_{E})$ and $(\mathbb{R}, \tau_{S})$.
\end{lemma}

\begin{proof}
Obviously, $D$ is a dense subset of $(\mathbb{R}, \tau_{E})$ since $\tau_{A}$ is finer than $\tau_{E}$. In order to prove $D$ being dense in $(\mathbb{R}, \tau_{S})$, take an arbitrary non-empty open subset $U$ in $\tau_{S}$, then there exists a non-empty open subset $V$ in $\tau_{E}$, hence $\emptyset\neq V\cap D\subset U\cap D$. Therefore, $D$ is also dense in $(\mathbb{R}, \tau_{S})$.
\end{proof}

By Lemma~\ref{l0}, we have the following corollary.

\begin{corollary}
For an arbitrary subset $A$ of $\mathbb{R}$, we have $d(\mathbb{R}, \tau_{A})=d(\mathbb{R}, \tau_{E})=d(\mathbb{R}, \tau_{S})=\omega$.
\end{corollary}

Since $(\mathbb{R}, \tau_{E})$ and $(\mathbb{R}, \tau_{S})$ are all Baire, we have the following corollary.

\begin{corollary}
For an arbitrary subset $A$ of $\mathbb{R}$, the $H$-space $(\mathbb{R}, \tau_{A})$ is a Baire space.
\end{corollary}

\begin{proposition}
For an arbitrary subset $A$ of $\mathbb{R}$, the $H$-space $(\mathbb{R}, \tau_{A})$ is homeomorphic to $(\mathbb{R}, \tau_{E})$ if and only if $A=\mathbb{R}$.
\end{proposition}

\begin{proof}
Assume that $(\mathbb{R}, \tau_{A})$ is homeomorphic to $(\mathbb{R}, \tau_{E})$ and $A\neq\mathbb{R}$. Hence $\mathbb{R}\setminus A\neq\emptyset$. Take an arbitrary point $a\in \mathbb{R}\setminus A$. Then $(-\infty, a)$ is an open and closed subset in $(\mathbb{R}, \tau_{A})$. Hence $(\mathbb{R}, \tau_{A})$ is not connected. However, $(\mathbb{R}, \tau_{E})$ is connected. Hence $A=\mathbb{R}$.
\end{proof}

Now, we can prove one of main results of this paper, which gives a characterization of subset $\mathbb{R}\setminus A$ such that $(\mathbb{R}, \tau_{A})$ is zero-dimensional.

\begin{theorem}
For an arbitrary subset $A$ of $\mathbb{R}$, the $H$-space $(\mathbb{R}, \tau_{A})$ is zero-dimensional if and only if $\mathbb{R}\setminus A$ is dense in $(\mathbb{R}, \tau_{E})$.
\end{theorem}

\begin{proof}
Necessity. Let $(\mathbb{R}, \tau_{A})$ be zero-dimensional. Assume that $\mathbb{R}\setminus A$ is not dense in $(\mathbb{R}, \tau_{E})$, then it follows from Lemma~\ref{l0} that $\mathbb{R}\setminus A$ is also not dense in $(\mathbb{R}, \tau_{A})$. Then there exists an open subset $U$ in $(\mathbb{R}, \tau_{A})$ such that $U\cap (\mathbb{R}\setminus A)=\emptyset$. Hence $U\subset A$, which implies that there exists an open interval $(c, d)$ such that $(c, d)\subset U\subset A$. Since $(\mathbb{R}, \tau_{A})$ is zero-dimensional and $(c, d)\subset A$, it follows that $(c, d)$ contains an open and closed subset $V$. Since each neighborhood of each point of $A$ belongs to $\tau_{E}$, it is easy to see that $V$ is an open and closed subset in $(\mathbb{R}, \tau_{E})$. However, $(\mathbb{R}, \tau_{E})$ is not zero-dimensional, which is a contradiction.

Sufficiency. Let $\mathbb{R}\setminus A$ be dense in $(\mathbb{R}, \tau_{E})$. Take an arbitrary $x_{0}\in\mathbb{R}$. We divide into the proof into the following two cases.

\smallskip
{\bf Case 1:} $x_{0} \in A$.

\smallskip
By Lemma~\ref{l0}, there exist a strictly increasing sequence $\{y_{n}\}$ and a strictly decreasing sequence $\{z_{n}\}$ such that $\{y_{n}\}\subset \mathbb{R}\setminus A$, $\{z_{n}\}\subset \mathbb{R}\setminus A$ and two sequences all converge to $x_{0}$ in $(\mathbb{R}, \tau_{E})$. For any $n\in\mathbb{N}$, put $U_{n}=[y_{n}, z_{n})$. Clearly, each $U_{n}$ is an open and closed subset of $(\mathbb{R}, \tau_{A})$. However, it easily check that the family $\{U_{n}: n\in\mathbb{N}\}$ is a base at $x_{0}$ in $(\mathbb{R}, \tau_{A})$. Hence the point $x_{0}$ has a base consisting of open and closed subsets in $(\mathbb{R}, \tau_{A})$.

\smallskip
{\bf Case 2:} $x_{0} \not\in A$.

\smallskip
Obviously, there exists a strictly decreasing sequence $\{x_{n}\}$ such that $\{x_{n}\}\subset \mathbb{R}\setminus A$ and $\{x_{n}\}$ converges to $x_{0}$ in $(\mathbb{R}, \tau_{E})$. Then the family $\{[x_{0}, x_{n}): n\in\mathbb{N}\}$ is base consisting of open and closed subsets in $(\mathbb{R}, \tau_{A})$.

In a word, $(\mathbb{R}, \tau_{A})$ is zero-dimensional.
\end{proof}

Next we prove the second main result of this paper, which shows that local compactness is equivalent to $k_{\omega}$ property in $(\mathbb{R}, \tau_{A})$. Indeed, A. Bouziad and E. Sukhacheva in \cite{BS} has proved that, for an arbitrary subset $A$ of $\mathbb{R}$, we have $(\mathbb{R}, \tau_{A})$ is locally compact if and only if $\mathbb{R}\setminus A$ is closed in $\mathbb{R}$ and discrete in $\mathbb{S}$.
\begin{theorem}\label{t2}
For an arbitrary subset $A$ of $\mathbb{R}$, then the following statements are equivalent:
\begin{enumerate}
\item $(\mathbb{R}, \tau_{A})$ is locally compact;

\item $(\mathbb{R}, \tau_{A})$ is a $k_{\omega}$-space;

\item $\mathbb{R}\setminus A$ is discrete and closed in $(\mathbb{R}, \tau_{A})$.
\end{enumerate}
\end{theorem}

\begin{proof}
From \cite{BS}, we have (1) $\Leftrightarrow$ (3). The implication (1) $\Rightarrow$ (2) is easily to be checked. It suffices to prove that (2) $\Rightarrow$ (1).

(2) $\Rightarrow$ (1). Assume that $(\mathbb{R}, \tau_{A})$ is a $k_{\omega}$-space, and let $\{K_{n}\}$ be an increasing sequence of compact subsets such that the family $\{K_{n}\}$ determines the topology of  $(\mathbb{R}, \tau_{A})$. Take an arbitrary $x_{0}\in\mathbb{R}$. Since $(\mathbb{R}, \tau_{A})$ is first-countable, choose an open neighborhood base $\{U_{n}: n\in\mathbb{N}\}$ of point $x_{0}$ in $(\mathbb{R}, \tau_{A})$ such that $U_{n+1}\subset U_{n}$ for each $n\in\mathbb{N}$. We claim that there exist $m, p\in\mathbb{N}$ such that $U_{m}\subset K_{p}$. Suppose not, then for each $n\in\mathbb{N}$ it follows that $U_{n}\setminus K_{n}\neq\emptyset$, hence choose a point $x_{n}\in U_{n}\setminus K_{n}$. Put $K=\{x_{n}: n\in\mathbb{N}\}\cup \{x_{0}\}$. Then $K$ is a compact subset in $(\mathbb{R}, \tau_{A})$ and $K\setminus K_{n}\neq\emptyset$ for each $n\in\mathbb{N}$. However, since $(\mathbb{R}, \tau_{A})$ is a $k_{\omega}$-space, it easily check that there exists $n_{0}\in\mathbb{N}$ such that $K\subset K_{n_{0}}$, which is a contradiction.
\end{proof}

From Theorem~\ref{t2}, it natural to pose the following question.

\begin{question}
What subsets $A$ of $\mathbb{R}$ are $(\mathbb{R}, \tau_{A})$ being $\sigma$-compact?
\end{question}

Next we give some a partial answer to this question. First, we give some lemmas.

\begin{proposition}\label{p1}
For arbitrary $B$ of $(\mathbb{R}, \tau_{S})$, we have $nw(B)\geq |B|$. In particular, $w(B)\geq |B|$.
\end{proposition}

\begin{proof}
Let $\mathscr{P}$ be an arbitrary network of the subspace $B$ with $|\mathscr{P}|=nw(B)$. For each $x\in B$, let
$$\mathscr{B}_{x}=\{x\in P: P\in\mathscr{P}, P\subset [0, \frac{1}{n})\cap B\ \mbox{for some}\ n\in\mathbb{N}\}.$$ Then $\bigcup_{x\in B}\mathscr{B}_{x}\subset \mathscr{P}$ and is also a network of the subspace $B$. However, for any $x, y\in B$ with $x\neq y$, we have $\mathscr{B}_{x}\cap \mathscr{B}_{y}=\emptyset$, hence $nw(B)\geq |B|$.
\end{proof}

By Proposition~\ref{p1} and the separability of $(\mathbb{R}, \tau_{S})$, we have the following corollary.

\begin{corollary}\label{l1}\cite[Proposition 2.3]{Chatyrko2}
For any subset $X$ of $(\mathbb{R}, \tau_{S})$, $X$ is metrizable if and only if $X$ is countable.
\end{corollary}

\begin{lemma}\label{l9}
For an arbitrary subset $A$ of $\mathbb{R}$, $(\mathbb{R}, \tau_{A})$ is submetrizable.
\end{lemma}

\begin{proof}
Since $(\mathbb{R}, \tau_{A})$ and $\tau_E\subset \tau_A$, it follows that $(\mathbb{R}, \tau_{A})$ is submetrizable.
\end{proof}

\begin{lemma}\label{l2}
If $K$ is a compact subset of $(\mathbb{R}, \tau_{A})$, then $K\cap (\mathbb{R}\setminus A)$ is countable.
\end{lemma}

\begin{proof}
By Lemma~\ref{l9}, $K$ is metrizable. Put $X=K\cap (\mathbb{R}\setminus A)$. Then $X$ is metrizable. Moreover, $X$ is subspace of $(\mathbb{R}, \tau_{S})$. By Corollary~\ref{l1}, $X$ is countable. Therefore, $K\cap (\mathbb{R}\setminus A)$ is countable.
\end{proof}

\begin{lemma}\label{t1}
For an arbitrary subset $A$ of $\mathbb{R}$, if $\overline{\mathbb{R}\setminus A}$ is countable then $(\mathbb{R}, \tau_{A})$ is $\sigma$-compact.
\end{lemma}

\begin{proof}
Put $U=\mathbb{R}\setminus\overline{\mathbb{R}\setminus A}$. Then $U$ is open in $(\mathbb{R}, \tau_{A})$ and $U\subset A$, hence $U$ is open in $(\mathbb{R}, \tau_{E})$, which implies that $U$ is $\sigma$-compact. From $U\subset A$, it follows that $U$ is $\sigma$-compact $(\mathbb{R}, \tau_{A})$. By the countability of $\overline{\mathbb{R}\setminus A}$, $(\mathbb{R}, \tau_{A})$ is $\sigma$-compact.
\end{proof}

Now we have the following two results.

\begin{theorem}\label{t4}
For an arbitrary subset $A$ of $\mathbb{R}$, $(\mathbb{R}, \tau_{A})$ is $\sigma$-compact, then $\mathbb{R}\setminus A$ is countable and nowhere dense in $(\mathbb{R}, \tau_{A})$.
\end{theorem}

\begin{proof}
By Lemma~\ref{l2}, it is easy to see that $\mathbb{R}\setminus A$ is countable.  It suffices to prove that $\mathbb{R}\setminus A$ is nowhere dense.

Assume that $\mathbb{R}\setminus A$ is not nowhere dense. Then there exists an open subset $V$ being contained in the closure of $\mathbb{R}\setminus A$. Hence there exist $a, b\in \mathbb{R}\setminus A$ such that $[a, b)\subset V$. Then $[a, b)$ is $\sigma$-compact since $[a, b)$ is open and closed in $(\mathbb{R}, \tau_{A})$, hence there exists a sequence of compact subsets $\{K_{n}\}$ of $(\mathbb{R}, \tau_{A})$ such that $[a, b)=\bigcup_{n\in\mathbb{N}} K_{n}$. By Lemma~\ref{l0}, it easily check that $[a, b)$ is a Baire space, then there exists $n\in\mathbb{N}$ such that $K_{n}$ contains an empty open subset $W$ in $(\mathbb{R}, \tau_{A})$. Since $W\subset [a, b)\subset V$, there exist $c, d\in\mathbb{R}\setminus A$ such that $[c, d)\subset [a, b)$. Since $[c, d)$ is closed and $[c, d)\subset K_{n}$, it follows that $[c, d)$ is compact, which is a contradiction. Therefore, $\mathbb{R}\setminus A$ is nowhere dense.
\end{proof}

\begin{theorem}\label{t6}
For an arbitrary subset $A$ of $\mathbb{R}$, if $\mathbb{R}\setminus A$ is countable and scattered in $(\mathbb{R}, \tau_{A})$, then $(\mathbb{R}, \tau_{A})$ is $\sigma$-compact.
\end{theorem}

\begin{proof}
Assume that $\mathbb{R}\setminus A$ is countable and scattered. Then it follows from \cite[Corollary 3]{V. Kannan} that $\mathbb{R}\setminus A$ is homeomorphic to a subspace of $[0, \omega_{1}).$ Hence it easily see that the closure of $\mathbb{R}\setminus A$ in $(\mathbb{R}, \tau_{A})$ is countable. By Lemma~\ref{t1}, $(\mathbb{R}, \tau_{A})$ is $\sigma$-compact.
\end{proof}

The following example shows that the property of $\sigma$-compact in $(\mathbb{R}, \tau_{A})$ dose not imply local compactness.

\begin{example}
There exists a subset $A$ such that $(\mathbb{R}, \tau_{A})$ is $\sigma$-compact but not locally compact.
\end{example}

\begin{proof}
Let $A=\mathbb{R}\setminus(\{0\}\cup\{\frac{1}{n}: n\in\mathbb{N})$. Then $\mathbb{R}\setminus A=\{0\}\cup\{\frac{1}{n}: n\in\mathbb{N}\}$ is closed, countable scattered and non-discrete. By Theorem~\ref{t6}, $(\mathbb{R}, \tau_{A})$ is $\sigma$-compact. However, it follows from Theorem~\ref{t2} that $(\mathbb{R}, \tau_{A})$ is not locally compact and not a $k_{\omega}$-space.
\end{proof}

Next we prove that $(\mathbb{R}, \tau_{A})^{\aleph_{0}}$ is perfectly subparacompact for arbitrary subset $A\subset\mathbb{R}$. The proof of the following theorem is similar to the proof of \cite[Lemma~2.3]{Heath}. However, for the convenience to the reader, we give out the proof.

\begin{theorem}\label{t5}
For an arbitrary subset $A$ of $\mathbb{R}$, $(\mathbb{R}, \tau_{A})^{n}$ is perfect for every $n\in\mathbb{N}$.
\end{theorem}

\begin{proof}
By induction. The theorem is clear for $n=1$ since  $(\mathbb{R}, \tau_{A})$ is a Lindel\"{o}f space. Therefore let us suppose the theorem for $n$ and let us prove it for $n+1$.

Let $Z=\prod_{i=1}^{n+1}Z_{i}$ with $Z_{i}=(\mathbb{R}, \tau_{A})$ for all $i\leq n+1.$ For every $m\leq n+1$, put $Z(m)=\prod_{i=1}^{n+1}Z_{i}(m)$, where $Z_{m}(m)=(\mathbb{R}, \tau_{E})$ and $Z_{i}(m)=(\mathbb{R}, \tau_{A})$ if $i\neq m$.

Now it suffices to prove that arbitrary open subset $U$ of $Z$ is an $F_{\sigma}$ in $Z$. For every $m\leq n+1$, let $U(m)$ be the interior of $U$ as a subset of $Z(m)$, and let $U^{\star}=\bigcup_{m=1}^{n+1}U(m)$. It follows from \cite[Lemma 2.2]{Heath} that each $Z(m)$ is perfect, hence $U(m)$ is an $F_{\sigma}$ in $Z(m)$ and thus also in $Z$. Therefore, $U^{\star}$ is also an $F_{\sigma}$ in $Z$. Put $U^{\prime}=U\setminus U^{\star}$. Thus it only remains to prove that $U^{\prime}$ is an $F_{\sigma}$ in $Z$. Clearly, for each $x=(x_{1}, \ldots, x_{n+1})\in U^{\prime}$, it has $x_{i}\in \mathbb{R}\setminus A$ for each $i\leq n+1$.

For each $z\in U^{\prime}$, let $\{W_{j}(z): j\in\mathbb{N}\}$ denote the base of neighborhoods of $z$ in $Z$ defined by
$$W_{j}(z)=\{y\in Z: y_{i}\in [z_{i}, z_{i}+\frac{1}{j})\ \mbox{for each}\ i\leq n+1\}.$$
For every $j\in\mathbb{N}$, let $$U^{\prime}_{j}=\{z\in U^{\prime}: W_{j}(z)\subset U\}.$$It easily see that $U^{\prime}=\bigcup_{j=1}^{\infty}U^{\prime}_{j}$.
Next we shall prove that each $U^{\prime}_{j}$ is closed in $Z$. Take an arbitrary $j\in\mathbb{N}$, assume $z\not\in U^{\prime}_{j}$, it suffices to prove that $z$ is not in the closure of $U^{\prime}_{j}$ in $Z$.

For each $F\subset\{1, \cdots, n+1\}$, let
$$U^{\prime}_{j, F}(z)=\{y\in U^{\prime}_{j}: z_{i}=y_{i}\ \mbox{iff}\ i\in F\}.$$
Clearly, $U^{\prime}_{j}=\bigcup\{U^{\prime}_{j, F}(z): F\subset \{1, \cdots, n+1\}\}$. Then it suffices to prove that for each $F\subset\{1, \cdots, n+1\}$ there exists a neighborhood of $z$ in $Z$ disjoint from $U^{\prime}_{j, F}$. Indeed, suppose that $W_{j}(z)\cap U^{\prime}_{j, F}(z)\neq\emptyset$, then it can choose a point $x\in W_{j}(z)\cap U^{\prime}_{j, F}(z)$. Then the set $$V=W_{j}(z)\cap\{y\in Z: y_{i}<x_{i}\ \mbox{if}\ i\not\in F\}$$ is a neighborhood of $z$ in $Z$, and it will suffice to prove that $V\cap U^{\prime}_{j, F}=\emptyset.$

Suppose not, then there exists some $y\in V\cap U^{\prime}_{j, F}$. Clearly, $y\in W_{j}(z)$ and $y\neq z$, thus there is an $m\leq n+1$ such that $y_{m}>z_{m}$. Then $m\not\in F$. Put $$W=W_{j}(y)\cap \{u\in Z: y_{m}<u_{m}\}.$$ Clearly, $W$ is open in $Z(m)$ and $W\subset W_{j}(y)\subset U$. It follows from the definition of $U(m)$ that $W\subset U(m)\subset U^{\star}$. Moreover, it easily check that $x\in W$. Therefore, $x\in U^{\star}$, which is a contradiction. That completes the proof.
\end{proof}

\begin{theorem}
For arbitrary $n\in\mathbb{N}$, $(\mathbb{R}, \tau_{A})^{n}$ is perfectly subparacompact.
\end{theorem}

\begin{proof}
By Theorem~\ref{t5}, $(\mathbb{R}, \tau_{A})^{n}$ is perfect. It suffices to prove that $(\mathbb{R}, \tau_{A})^{n}$ is subparacompact.

By induction. The result is certainly true for $n=1$. Let us assume that $(\mathbb{R}, \tau_{A})^{n}$ is subparacompact. Next it will prove that $(\mathbb{R}, \tau_{A})^{n+1}$ is subparacompact. By \cite[Proposition 2.9]{Lutzer}, it suffices to prove that $(\mathbb{R}, \tau_{A})^{n+1}$ is weakly $\theta$-refinable. Put $Z=(\mathbb{R}, \tau_{A})^{n+1}$. Let $\mathscr{W}=\{W(\alpha): \alpha\in A\}$ be a basic open cover of the space $Z$, where $W(\alpha)=U(1, \alpha)\times \cdots\times U(n+1, \alpha)$ such that $U(k, \alpha)=(a(k, \alpha), b(k, \alpha))$ if $a(k, \alpha)\in A$ and $U(k, \alpha)=[a(k, \alpha), b(k, \alpha))$ if $a(k, \alpha)\not\in A$. By the same notations in the proof of Theorem~\ref{t5}, let $W(\alpha, m)$ be the interior of the set $W(\alpha)$ in $Z(m)$ for each $m\leq n+1$, thus $W(\alpha, m)$ is open in $Z(m)$, hence also open in $Z$. For each $m\leq n+1$, put $\mathscr{G}(m)=\{W(\alpha, m): \alpha\in A\}$. By \cite[Corollary 2.6]{Lutzer}, each $Z(m)$ is perfect subparacompact. Then it follows from \cite[Porposition 2.9]{Lutzer} that $\mathscr{G}(m)$ has a weakly $\theta$-refinement $\mathscr{H}(m)$ which covers $\bigcup\mathscr{H}(m)$ and which consists of open subsets of $Z(m)$. Clearly, $\mathscr{H}(m)$ is also a collection of open subsets of $Z$. Let $$Y=Z\setminus\cup\{\bigcup\mathscr{H}(m): 1\leq m\leq n+1\}.$$ Then for each $y\in Y$ it has $y_{i}\in \mathbb{R}\setminus A$ for each $i\leq n+1$, hence there exists $\alpha\in A$ such that $y_{i}=a(i, \alpha)$ for each $i\leq n+1$. For each $y\in Y$, pick $\alpha(y)\in A$ such that $y\in W(\alpha(y))$. Put $$\mathscr{H}(0)=\{W(\alpha(y)): y\in Y\}.$$ It easily see that if $x$ and $y$ are distinct elements of $Y$, then $y\not\in W(\alpha(x))$. Therefore, $\mathscr{H}(0)$ is a collection of open subsets of $Z$ which covers $Y$ in such a way that each point of $Y$ belongs to exactly one member of $\mathscr{H}(0)$. Hence $\mathscr{H}=\{\mathscr{H}(m): m\in\omega\}$ is a weak $\theta$-refinement of $\mathscr{W}$. Therefore, $Z$ is weakly $\theta$-refinable.
\end{proof}

By \cite[Proposition 2.1]{Heath} and \cite[Proposition 2.7]{Lutzer}, we have the following theorem.

\begin{theorem}
For an arbitrary subset $A$ of $\mathbb{R}$, $(\mathbb{R}, \tau_{A})^{\aleph_{0}}$ is perfectly subparacompact.
\end{theorem}

\begin{corollary}\cite{Lutzer,Heath}
The space $(\mathbb{R}, \tau_{S})^{\aleph_{0}}$ is perfectly subparacompact.
\end{corollary}

Finally, we consider the quasi-metrizability of $H$-spaces. It is well-known that $(\mathbb{R}, \tau_{E})$ and $(\mathbb{R}, \tau_{S})$ are all quasi-metrizable, it natural to pose the following question.

\begin{question}\label{q11}
For an arbitrary $A\subset \mathbb{R}$, is $(\mathbb{R}, \tau_{A})$ quasi-metrizable?
\end{question}

We give some a negative answer to Question~\ref{q11} in Example~\ref{e111}. Indeed, from the definition of generalized ordered space, we have the following proposition.

\begin{proposition}\label{p0}
For arbitrary $A\subset\mathbb{R}$, the $H$-space $(\mathbb{R}, \tau_{A})$ is a generalized ordered space.
\end{proposition}

By \cite[Theorem 10]{Jacob Kolner}, we can easily give a characterization of subset $A$ of $\mathbb{R}$ such that $(\mathbb{R}, \tau_{A})$ is quasi-metrizable, see Theorem~\ref{t4545}.

\begin{theorem}\label{t4545}
For any subset $A\subset \mathbb{R}$, the $H$-space $(\mathbb{R}, \tau_{A})$ is quasi-metrizable if and only if $\mathbb{R}\setminus A$ is a $F_{\sigma}$-set in $(\mathbb{R}, \tau_{S^{-}})$, where $(\mathbb{R}, \tau_{S^{-}})$ is the set of real numbers with the topology generated by the base $\{(a, b]: a, b\in\mathbb{R}, a<b\}$.
\end{theorem}

Now, we can give a negative answer to Question~\ref{q11}.

\begin{example}\label{e111}
There exists a subset $A$ of $\mathbb{R}$ such that $(\mathbb{R}, \tau_{A})$ is not quasi-metrizable.
\end{example}

\begin{proof}
Indeed, let $A=\mathbb{Q}$ be the rational number. By Theorem~\ref{t4545}, assume $\mathbb{R}\setminus A$ is a $F_{\sigma}$-set in $(\mathbb{R}, \tau_{S^{-}})$, then $\mathbb{Q}$ is a $G_{\delta}$-set in $(\mathbb{R}, \tau_{S^{-}})$. However, it follows from \cite[Theorem 3.4]{D.K. Burke1} that
$(\mathbb{R}, \tau_{S^{-}})$ does not have a dense metrizable $G_{\delta}$-space, which is a contradiction.
\end{proof}

Obviously, if $\mathbb{R}\setminus A$ is a $F_{\sigma}$-set in $(\mathbb{R}, \tau_{A})$, then $\mathbb{R}\setminus A$ is a $F_{\sigma}$-set in $(\mathbb{R}, \tau_{S^{-}})$, hence we have the following corollary.

\begin{corollary}\label{cc0000}
If $\mathbb{R}\setminus A$ is a $F_{\sigma}$-set in $(\mathbb{R}, \tau_{A})$, then $(\mathbb{R}, \tau_{A})$ is quasi-metrizable.
\end{corollary}

We now close this section with a result about generalized metric property of $H$-space.

\begin{theorem}\label{t111}
For an arbitrary $A\subset\mathbb{R}$, then the following statements are equivalent:
\begin{enumerate}
\item $(\mathbb{R}, \tau_{A})$ is metrizable;
\item $(\mathbb{R}, \tau_{A})$ is a $\beta$-space;
\item $\mathbb{R}\setminus A$ is countable.
\end{enumerate}
\end{theorem}

\begin{proof}
Obviously, it suffices to prove (2) $\Rightarrow$ (3). Assume that $(\mathbb{R}, \tau_{A})$ is a $\beta$-space. Since $(\mathbb{R}, \tau_{A})$ is a paracompact submetrizable space, it follows from \cite[Theorem 7.8 (ii)]{Gr} that $(\mathbb{R}, \tau_{A})$ is semi-stratifiable. By Proposition~\ref{p0} and \cite[Theorems 5.16 and 5.21]{Gr}, $(\mathbb{R}, \tau_{A})$ is a stratifiable space, hence $(\mathbb{R}, \tau_{A})$ is a $\sigma$-space by \cite[Theorem 5.9]{Gr}. Then $(\mathbb{R}, \tau_{A})$ has a countable network since $(\mathbb{R}, \tau_{A})$ is separable, hence $\mathbb{R}\setminus A$ has a countable network. Therefore, it follows from Proposition~\ref{p1} that $\mathbb{R}\setminus A$ must be countable.
\end{proof}

\maketitle
\section{Open questions}
It is well known that $(\mathbb{R}, \tau_{E})\times (\mathbb{R}, \tau_{E})$ is Lindel\"{o}f, and $(\mathbb{R}, \tau_{S})\times (\mathbb{R}, \tau_{S})$ is not Lindel\"{o}f, hence it is natural to pose the following question.

\begin{question}\label{q0}
For an arbitrary subset $A$ of $\mathbb{R}$, are the following statements equivalent?
\begin{enumerate}
\item $(\mathbb{R}, \tau_{A})\times (\mathbb{R}, \tau_{A})$ is Lindel\"{o}f;

\item $(\mathbb{R}, \tau_{A})\times (\mathbb{R}, \tau_{A})$ is normal;

\item $(\mathbb{R}, \tau_{A})$ is metizable.
\end{enumerate}
\end{question}

The following example gives a negative answer to Question~\ref{q0} under the assumption of CH.

\begin{example}
Under the assumption of CH, there exists a subspace $A\subset\mathbb{R}$ such that $\mathbb{R}\setminus A$ is uncountable and $(\mathbb{R}, \tau_{A})\times (\mathbb{R}, \tau_{A})$ is Lindel\"{o}f.
\end{example}

\begin{proof}
By \cite[Theorem 3.4]{D.K. Burke}, there exists an uncountable subset $Y\subset \mathbb{S}$ such that $Y^{2}$ is Lindel\"{o}f. Put $A=\mathbb{R}\setminus Y$. Then $(\mathbb{R}, \tau_{A})\times (\mathbb{R}, \tau_{A})$ is Lindel\"{o}f. Indeed, it is obvious that $$(\mathbb{R}, \tau_{A})\times (\mathbb{R}, \tau_{A})=(A\times A)\cup (A\times Y)\cup(Y\times A)\cup(Y\times Y).$$ Since $A$ is a separable metrizabale space, the subspace $A\times A$, $A\times Y$ and $Y\times A$ are Lindel\"{o}f. Therefore, $(\mathbb{R}, \tau_{A})\times (\mathbb{R}, \tau_{A})$ is Lindel\"{o}f.
\end{proof}

By Theorem~\ref{t6}, we have the following question.

\begin{question}
If $(\mathbb{R}, \tau_{A})$ is $\sigma$-compact, is $A$ a scattered subspace?
\end{question}

The following question was posed by Boaz Tsaban.

\begin{question}
When is the space $(\mathbb{R}, \tau_{A})$ Menger (Hurewicz) for any $A\subset\mathbb{R}$?
\end{question}

\end{document}